\documentclass{amsart}

\usepackage{array}
\usepackage{enumerate, url, amsaddr}

\newcommand{\NN}{\mathbb{N}}
\newcommand{\PP}{\mathbb{P}}
\newcommand{\ZZ}{\mathbb{Z}}

\renewcommand{\bold}[1]{\mathbf{#1}}

\begin{document}

\title[Cond'l Prob. Multivariate Poisson]{Conditional Probabilities of Multivariate Poisson Distributions}
\author{Michael C. Burkhart}
\address{Rutgers University Department of Mathematics\\
Hill Center for the Mathematical Sciences  \\
110 Frelinghuysen Rd, Piscataway, NJ 08854-8019}



\theoremstyle{plain}
\newtheorem*{thm}{Theorem}
\newtheorem*{lem}{Lemma}
\newtheorem{prop}{Proposition}
\newtheorem*{cor}{Corollary}
\newtheorem{methd}{Method}

\theoremstyle{definition}
\newtheorem*{defn}{Definition}
\newtheorem*{conj}{Conjecture}
\newtheorem*{exmp}{Example}

\theoremstyle{remark}
\newtheorem*{rem}{Remark}
\newtheorem*{note}{Note}
\newtheorem*{case}{Case}


\begin{abstract}
Multivariate Poisson distributions have numerous applications.  Fast computation of these distributions, holding constant a fixed set of linear combinations of these variables, has been explored by Sontag and Zeilberger \cite{SZ10}.  This elaborates on their work.  
\end{abstract}

\maketitle


\section{Introduction}
Set $\NN = \ZZ_{\ge 0}$.  Let $ X_1,\dotsc,X_n$ be independent Poisson random variables with $X_i \sim Poisson(\lambda_i)$.  Let $\bold A \in \bold{Mat}_{m\times n}(\NN)$.  Define new random variables $Y_1, \dotsc, Y_m$ by taking linear combinations of the $X_i$:
\[
\left(\begin{smallmatrix} Y_1 \\ \vdots \\ Y_m \end{smallmatrix} \right)
= \bold A \left(\begin{smallmatrix} X_1 \\ \vdots \\ X_n \end{smallmatrix} \right)
\]

Then the joint p.d.f for $Y_1,\dotsc,Y_m$ is given:
\begin{multline*}
F_{\bold Y} (b_1,\dotsc,b_m)
= \PP(Y_1=b_1,\dotsc,Y_m=b_m) \\
= \sum_{\substack{(k_1,\dotsc,k_n)\in\NN^n \\ (b_1,\dotsc,b_m)^T=\bold A(k_1,\dotsc,k_n)^T}}
\frac{\lambda_1^{k_1}}{k_1!}\frac{\lambda_2^{k_2}}{k_2!}\dotsb \frac{\lambda_n^{k_n}}{k_n!} \ e^{-(\lambda_1+\lambda_2+\dotsb+\lambda_n)}
\end{multline*}

Sontag and Zeilberger demonstrated that the multivariate generating function of $\bold Y$:
\[
G_{\bold Y}(z_1,\dotsc, z_m)
= \sum_{(b_1,\dotsc, b_m) \in\NN^m} 
F_{\bold Y}(b_1,\dotsc,b_m)\ z_1^{b_1} \dotsb z_m^{b_m}
\]
may be simplified  \cite{SZ10} (\emph{cf.} \cite{Bur12}):
\[
G_{\bold Y}(z_1,\dotsc, z_m) = \exp\left(\sum_{j=1}^n \lambda_j \prod_{i=1}^m z_i^{a_{ij}}\right)
\]

Certain conditions can be placed on the matrix $\bold A$ to guarantee that the joint p.d.f. $F_{\bold Y}$ can be expressed as the sum over a single index (and thus practicably computed with Wilf-Zeilberger recurrences \cite{WZ92, AZ06}).  In particular, Sontag and Zeilberger established that $\bold A \in \bold{Mat}_{2\times n}(\{0,1\})$ is a sufficient criterion for reduction to a single-index sum \cite{SZ10}.  This paper will explore another sufficient criterion on the matrix $\bold A$ to guarantee that the joint p.d.f. can be expressed as the sum over a single index.

\section{Smith Normal Form}
Let $\bold A \in \bold{Mat}_{m\times n} (\NN)$.  Due to the nature of the problem, it may be assumed without loss that the rows of $\bold A$ are linearly independent.  In particular, this imposes $m\leq n$.

Any matrix $\bold A \in \bold{Mat}_{m\times n} (\ZZ)$ can be decomposed into its Smith normal form \cite[pp. 181-185]{Jac1}: i.e.: there exist invertible (over $\ZZ$) matrices $\bold P \in \bold{Mat}_{m\times m} (\ZZ)$ and $\bold Q \in \bold{Mat}_{n\times n} (\ZZ)$ such that:
\[
\left(\begin{smallmatrix} d_1 & & & & &  \\ & \ddots & & & & \\ & & & d_r & & \\ & & & & 0 & \\ & & & & & \ddots \end{smallmatrix}\right)_{m\times n}= \bold{PAQ}
\]
where $r$ is the rank of $\bold A$, $d_1, \dotsc, d_r \ne 0$, and $d_1 \mid \dotsb \mid d_r$.  These $d_i$, $1\leq i \leq r$ are unique up to sign-change and are called elementary divisors.  For $1\leq i \leq r$, let $\Delta_i$ denote the g.c.d. of the $i$-rowed minors of $\bold A$.  Then $d_1=\pm \Delta_1$, $d_2= \pm \Delta_2^{} \Delta_1^{-1}$, \dots, $ d_r=\pm \Delta_r ^{}\Delta_{r-1}^{-1}$.

\begin{thm}
If the matrix $\bold A\in \bold{Mat}_{m\times n}(\NN)$ has rank $r=m=n-1$ and elementary divisors $d_i=\pm 1$ for all $1\leq i \leq r$, then $F_{\bold Y}(k_1,\dotsc,k_m)$ can be written as the sum over a single index.
\end{thm}

\begin{proof}
Note that:
\[
F_{\bold Y}(b_1,\dotsc,b_m)
= \sum_{\substack{(k_1,\dotsc, k_n) \in\NN^n \\ (b_1,\dotsc,b_m)^T=\bold A(k_1,\dotsc,k_n)^T }} 
\frac{\lambda_1^{k_1}}{k_1!}\frac{\lambda_2^{k_2}}{k_2!}\dotsb \frac{\lambda_n^{k_n}}{k_n!} \ e^{-(\lambda_1+\lambda_2+\dotsb+\lambda_n)}
\]
By hypothesis, there exist invertible matrices $\bold P \in \bold{Mat}_{n-1\times n-1} (\ZZ)$ and $\bold Q \in \bold{Mat}_{n\times n} (\ZZ)$ such that:
\[
\left(\bold I_{n-1} \big\vert  \bold 0_{n\times 1} \right) = \bold{PAQ}
\]
where $\bold I_{n-1}$ denotes the $(n-1)\times (n-1)$ identity matrix.  Then $(b_1,\dotsc,b_m)^T=\bold A(k_1,\dotsc,k_n)^T$ may be rewritten:
\[
\bold P^{-1} \left(\bold I_{n-1} \big\vert  \bold 0_{n\times 1} \right)  \bold Q^{-1} \left( \begin{matrix} k_1 \\ \vdots \\ k_n \end{matrix} \right)
=  \left( \begin{matrix} b_1 \\ \vdots \\ b_{m} \end{matrix} \right)
\]
 so that the first $n-1$ entries of the vector $\bold Q^{-1} (k_1,\dotsc,k_m)^T$ agree with the $(n-1)\times 1$ vector $\bold P (b_1,\dotsc,b_m)^T$.  In particular, by introducing $j\in\NN$ as the index of summation it follows that the full set of solutions to $\bold A (k_1,\dotsc, k_n)^T = (b_1,\dotsc,b_m)^T$ for $(k_1,\dotsc, k_n) \in\NN^n$ is given:
\[
\left( \begin{matrix} k_1 \\ \vdots \\ k_n \end{matrix} \right)
=  \bold{Q} \left( \begin{matrix} P_{11} & \dotsb & & P_{1n} \\ \vdots & & & \vdots \\ P_{m1} & \dotsb && P_{mn} \\ 0 & \dotsb & 0 &1 \end{matrix}  \right) \left( \begin{matrix} b_1 \\ \vdots \\ b_{n-1} \\ j \end{matrix} \right)
\]
Substituting $(k_1,\dotsc,k_n)$ as given above yields an expression for $F_{\bold Y}(b_1,\dotsc,b_m)$ as the summation over a single index $j$. 
\end{proof}

\begin{exmp}
Let:
\[
\bold A = \left(\begin{smallmatrix} 1 & 0 & 1 \\ 0 & 2 & 1 \end{smallmatrix}\right)
\]
Then:
\[
\left(\begin{smallmatrix} 1 & 0 & 0 \\ 0 & 1 & 0 \end{smallmatrix}\right) 
= \bold{PAQ} 
= \left(\begin{smallmatrix} 1 & 0 \\ 0 & 1 \end{smallmatrix}\right) \left(\begin{smallmatrix} 1 & 0 & 1 \\ 0 & 2 & 1 \end{smallmatrix}\right) \left(\begin{smallmatrix} 1 & -1 & 2 \\ 0 & 0 & 1  \\ 0 & 1 & -2 \end{smallmatrix}\right)
\]
So that $\bold A (k_1,\dotsc, k_n)^T = (b_1,\dotsc,b_m)^T$ has solutions:
\[
\left(\begin{smallmatrix} k_1 \\ k_2 \\  k_3 \end{smallmatrix}\right) 
= \bold {Q} \left(\begin{smallmatrix} \vec b \\ j \end{smallmatrix}\right) 
= \left(\begin{smallmatrix} 1 & -1 & 2 \\ 0 & 0 & 1  \\ 0 & 1 & -2 \end{smallmatrix}\right) \left(\begin{smallmatrix} b_1 \\  b_2 \\ j \end{smallmatrix}\right) 
=\left(\begin{smallmatrix} b_1-b_2+2j \\ j \\ b_2-2j \end{smallmatrix}\right)
\]
So that:

\begin{multline*}
F_{\bold Y}(b_1,b_2)
= \sum_{\substack {(k_1,k_2,k_3)\in\NN^3 \\ (b_1,b_2)^T=\bold A(k_1,k_2,k_3)^T}} 
\frac{\lambda_1^{k_1}}{k_1!}\frac{\lambda_2^{k_2}}{k_2!}\frac{\lambda_3^{k_3}}{k_3!} \ e^{-(\lambda_1+\lambda_2+\lambda_3)} \\
=\sum_{j} 
\frac{\lambda_1^{b_1-b_2+2j}\lambda_2^{j}\lambda_3^{b_2-2j}}{(b_1-b_2+2j)!j!(b_2-2j)!} e^{-(\lambda_1+\lambda_2+\lambda_3)}
\end{multline*}

\end{exmp}

\begin{exmp}
Let:
\[
\bold A = \left(\begin{smallmatrix} 1 & 3 & 2 & 2 \\ 5 & 16 & 12 & 17 \\ 3 & 16 & 21 & 56 \end{smallmatrix}\right)
\]
Then:
\[
\left(\begin{smallmatrix} 1 & 0 & 0 & 0 \\ 0 & 1 & 0 &0 \\ 0 & 0 & 1 &0 \end{smallmatrix}\right) 
= \bold{PAQ} 
= \left(\begin{smallmatrix} 1 & 0 & 0 \\ 0 & 1  & 0 \\ 0 & 0 & 1\end{smallmatrix}\right) \left(\begin{smallmatrix} 1 & 3 & 2 & 2 \\ 5 & 16 & 12 & 17 \\ 3 & 16 & 21 & 56 \end{smallmatrix}\right) \left(\begin{smallmatrix} 144 & -31 & 4 & 15 \\ -69 & 15 & -2 & -5  \\ 32 & -7 & 1 & -1 \\ 0 & 0 & 0 & 1 \end{smallmatrix}\right)
\]
So that $\bold A (k_1,\dotsc, k_4)^T = (b_1,\dotsc,b_3)^T$ has solutions:
\[
\vec k
= \left(\begin{smallmatrix} k_1 \\ k_2 \\  k_3 \\ k_4 \end{smallmatrix}\right) 
= \bold {Q} \left(\begin{smallmatrix} b_1 \\  b_2 \\ b_3 \\ j \end{smallmatrix}\right) 
= \left(\begin{smallmatrix} 144 & -31 & 4 & 15 \\ -69 & 15 & -2 & -5  \\ 32 & -7 & 1 & -1 \\ 0 & 0 & 0 & 1 \end{smallmatrix}\right) \left(\begin{smallmatrix} b_1 \\  b_2 \\ b_3 \\ j \end{smallmatrix}\right) 
=\left(\begin{smallmatrix} 144b_1-31b_2+4b_3+15j \\ -69b_1+15b_2-2b_3-5j \\ 32b_1-7b_2+b_3-j \\ j\end{smallmatrix}\right)
\]
So that:

\begin{multline*}
F_{\bold Y}(b_1,b_2,b_3)
= \sum_{\substack {(k_1,k_2,k_3,k_4)\in\NN^4 \\ (b_1,b_2,b_3)^T=\bold A(k_1,k_2,k_3,k_4)^T}} 
\frac{\lambda_1^{k_1}}{k_1!}\frac{\lambda_2^{k_2}}{k_2!}\frac{\lambda_3^{k_3}}{k_3!}\frac{\lambda_4^{k_4}}{k_4!} \ e^{-(\lambda_1+\lambda_2+\lambda_3+\lambda_4)} \\
=\sum_{j} 
\frac{\lambda_1^{144b_1-31b_2+4b_3+15j}\lambda_2^{-69b_1+15b_2-2b_3-5j}\lambda_3^{32b_1-7b_2+b_3-j}\lambda_4^j e^{-(\lambda_1+\lambda_2+\lambda_3+\lambda_4)}}{(144b_1-31b_2+4b_3+15j)!(-69b_1+15b_2-2b_3-5j)!(32b_1-7b_2+b_3-j)!j!} 
\end{multline*}

\end{exmp}

\section{Special Case: $\bold A$ is invertible}

\begin{prop}
If the matrix $\bold A\in \bold{Mat}_{n\times n}(\ZZ_{\ge 0})$ is invertible, then $F_{\bold Y}(\vec b)$ can be expressed explicitly.
\end{prop}

\begin{proof}
Note that:
\[
F_{\bold Y}(\vec b)
= \sum_{\substack {\vec k \in \bold{Mat}_{n\times 1}(\ZZ_{\ge 0}) \\ \bold A \vec k = \vec b}} 
\frac{\lambda_1^{k_1}}{k_1!}\frac{\lambda_2^{k_2}}{k_2!}\dotsb \frac{\lambda_n^{k_n}}{k_n!} \ e^{-(\lambda_1+\lambda_2+\dotsb+\lambda_n)}
\]
But this sum is taken over the singleton:
\[
\{\vec k \in \bold{Mat}_{n\times 1}(\ZZ_{\ge 0}) : \bold A \vec k = \vec b \} = \{\vec k \in \bold{Mat}_{n\times 1} : \vec k = A^{-1} \vec b\}
\]
\end{proof}

\begin{exmp}
Let:
\[
\bold A = \left(\begin{smallmatrix} 1 & 5 & 3 \\ 2 & 10 & 5 \\ 0 & 1 & 8 \end{smallmatrix}\right)
\]
Then $\det(\bold A)=1$ and for any $\vec b = (b_1,b_2,b_3)^T \in  \bold{Mat}_{3\times 1}(\ZZ_{\ge 0})$, the unique solution to $\bold A \vec k = \vec b$ for $\vec k \in \bold{Mat}_{n\times 1}(\ZZ_{\ge 0})$ is given:
\[
\vec k 
= \left(\begin{smallmatrix} k_1 \\ k_2 \\  k_3 \end{smallmatrix}\right) 
= \bold A^{-1} \vec b 
= \left(\begin{smallmatrix} 75 & -37 & -5 \\ -16 & 8 & 1 \\ 2 & -1 & 0 \end{smallmatrix}\right) \left(\begin{smallmatrix} b_1\\ b_2 \\ b_3 \end{smallmatrix}\right)
= \left(\begin{smallmatrix} 75b_1-37b_2-5b_3 \\ -16b_1+8b_2+b_3 \\ 2b_1-b_2 \end{smallmatrix}\right)
\]
Thus the explicit probability density function for $\bold Y$ is given:
\[
F_{\bold Y}(\vec b)
= \frac{\lambda_1^{75b_1-37b_2-5b_3}\lambda_2^{-16b_1+8b_2+b_3}\lambda_3^{2b_1-b_2} e^{-(\lambda_1+\lambda_2+\lambda_3)}}{(75b_1-37b_2-5b_3)!(-16b_1+8b_2+b_3)!(2b_1-b_2)!}
\]
\end{exmp}

\section*{Acknowledgements}
This paper would not have been possible without the support and guidance of Dr. Eduardo Sontag.  Thanks are also due to Dr. Doron Zeilberger.  Supported in part by grant AFOSR FA9550-11-1-0247.

\bibliographystyle{amsalpha}
\bibliography{SmithNormal}

\end{document}